\documentclass{amsart}
\usepackage[latin9]{inputenc}
\usepackage{amstext}
\usepackage{amsthm}
\usepackage{amssymb}
\usepackage{soul}
\usepackage{xfrac}
\usepackage{amsmath}
\usepackage{bm}
\usepackage{enumitem}
\usepackage[dvipsnames]{xcolor}

\makeatletter

\numberwithin{equation}{section}
\numberwithin{figure}{section}

 \usepackage{amsthm}\usepackage{amscd}\usepackage{amsxtra}\usepackage{amsfonts}\usepackage{color}

\makeatother

\begin{document}
\newtheorem{theorem}{Theorem}[section] \newtheorem{lemma}[theorem]{Lemma}
\newtheorem{definition}[theorem]{Definition} \newtheorem{example}[theorem]{Example}
\newtheorem{proposition}[theorem]{Proposition} \newtheorem{corollary}[theorem]{Corollary}
\newtheorem{conjecture}[theorem]{Conjecture} 
\newtheorem{remark}{Remark} 
\newtheorem{main}{Main Theorem}
\global\long\def\wta{{\rm {wt} } a }
 \global\long\def\R{\frak{R}}
\global\long\def\<{\langle}
\newcommand{\bea}{\begin{eqnarray}}
\newcommand{\eea}{\end{eqnarray}}
 \global\long\def\LL{\mathcal{L}}
 \global\long\def\>{\rangle}
 \global\long\def\t{\tau}
 \global\long\def\a{\alpha}
 \global\long\def\e{\epsilon}
 \global\long\def\l{\lambda}
 \global\long\def\ga{\gamma}
 \global\long\def\L{ L}
 \global\long\def\b{\beta}
 \global\long\def\om{\omega}
 \global\long\def\o{\omega}
 \global\long\def\c{\chi}
 \global\long\def\ch{\chi}
 \global\long\def\cg{\chi_{g}}
 \global\long\def\ag{\alpha_{g}}
 \global\long\def\ah{\alpha_{h}}
 \global\long\def\ph{\psi_{h}}
 \global\long\def\gi{\gamma}
 \global\long\def\supp{{\rm {supp}}}
 \global\long\def\GG{\mathcal{G}}
 \global\long\def\NN{\mathcal{N}}
 \global\long\def\wtb{{\rm {wt} } b }
 \global\long\def\bea{
\begin{eqnarray}
\end{eqnarray}
 }
 \global\long\def\eea{{eqnarray}}
 \global\long\def\be{
\begin{equation}
\end{equation}
 }
 \global\long\def\ee{{equation}}
 \global\long\def\g{\frak{g}}
 \global\long\def\tg{\tilde{\frak{g}} }
 \global\long\def\hg{\hat{\frak{g}} }
 \global\long\def\hb{\hat{\frak{b}} }
 \global\long\def\hn{\hat{\frak{n}} }
 \global\long\def\h{\frak{h}}
 \global\long\def\wt{{\rm {wt} } }
 \global\long\def\V{{\cal V}}
 \global\long\def\hh{\hat{\frak{h}} }
 \global\long\def\n{\frak{n}}
 \global\long\def\Z{\mathbb{Z}}
 \global\long\def\lar{\longrightarrow}
 \global\long\def\X{\mathfrak{X}}

\global\long\def\Zp{{\Bbb Z}_{\ge0} }

\global\long\def\N{\mathbb{N}}
 \global\long\def\C{\mathbb{C}}
 \global\long\def\Q{\Bbb Q}
 \global\long\def\WW{\boldsymbol{\mathcal{W}}}
\global\long\def\gl{\mathfrak{gl}}
 \global\long\def\sl{\mathfrak{sl}}
 \global\long\def\la{\langle}
 \global\long\def\ra{\rangle}
 \global\long\def\bb{\mathfrak{b}}
 \global\long\def\triplet{\mathcal{W}(p)}
 \global\long\def\striplet{\mathcal{SW}(m)}

\global\long\def\nordplus{\mbox{\scriptsize\ensuremath{{+\atop +}}}}
 \global\long\def\ds{{\displaystyle }}
\global\long\def\halmos{\rule{1ex}{1.4ex}}
 \global\long\def\pfbox{\hspace*{\fill}\mbox{\ensuremath{\halmos}}}
\global\long\def\epfv{\hspace*{\fill}\mbox{\ensuremath{\halmos}}}
 \global\long\def\nb{{\mbox{\tiny{ \ensuremath{{\bullet\atop \bullet}}}}}}
\global\long\def\nn{\nonumber}

\global\long\def\HH{\widetilde{\mathcal{T}}}
 \global\long\def\vak{{\bf 1}}

\global\long\def\sect#1{\section{#1}\setcounter{equation}{0}\setcounter{rema}{0}}
 \global\long\def\ssect#1{\subsection{#1}}
 \global\long\def\sssect#1{\subsubsection{#1}}


\title[]{ On  irreducibility of  modules of Whittaker type for cyclic orbifold vertex algebras }

\author{Dra\v{z}en Adamovi\'c, Ching Hung Lam,  Veronika Pedi\' c and Nina Yu}

\begin{abstract}
We extend the Dong-Mason theorem on irreducibility of modules for   orbifold vertex algebras  (cf.
\cite{DM})  to the category of weak modules.
Let $V$ be a vertex operator algebra,  $g$ an automorphism of order $p$. Let
$W$ be an irreducible weak $V$--module such that
$W,W\circ g,\dots,W\circ g^{p-1}$ are inequivalent irreducible modules.
We prove that $W$ is an   irreducible  weak $V^{\left\langle g\right\rangle }$\textendash module.
This result can be applied  on irreducible modules of  certain Lie algebra $\mathfrak L$ such that
$W,W\circ g,\dots,W\circ g^{p-1}$ are Whittaker modules having different Whittaker functions. We
present certain applications in the cases of the Heisenberg and Weyl vertex operator algebras. 
\end{abstract}


\maketitle

\section{Introduction} The study of quantum Galois theory for vertex operator algebras has been
initiated by C. Dong and G. Mason in \cite{DM}, and this theory has many applications in the vertex
operator algebra theory. Let us discuss one such application. Let $g$ be an automorphism of finite
order, and $V^{\left\langle g\right\rangle }$ be its subalgebra of fixed points under $g$. Let $M$ be a module for $V$. In
\cite{DM}, the authors discuss the structure of $M$ as a $V^{\left\langle g\right\rangle }$--module. In particular, they prove
that if $M$ is an irreducible ordinary $V$--module such that it is not isomorphic to $M\circ
g^{i}$, for all $i$,  then $M$ is also an irreducible module for the subalgebra $V^{\left\langle g\right\rangle }$.  This result
is important for the construction of irreducible modules for vertex algebras $M(1) ^+$ and $V_L^+$
(cf.\cite{D1,D2,DN1,DN2,A1,DN3,A2,AD,ADL,DJL}). Recently, it has also been applied for the
realization of irreducible modules for subalgebras of the triplet vertex algebra $\mathcal W(p)$
(cf. \cite{ALM}). The proof in \cite{DM} is based on certain applications of Zhu's algebra theory.

The question then arises as to whether the statement of Dong-Mason theorem is true for any weak
$V$--module.  It is clear that the original Dong--Mason proof cannot be applied to arbitrary weak
modules, since we cannot apply Zhu's algebra theory.  In this paper, we present an extension of
Dong-Mason theorem for weak modules of Whittaker type.

Recently, Whittaker modules have been investigated in the framework of vertex operator algebra
theory  in \cite{ALZ,HY,T1}. The latter two present a Lie theoretic proof of irreducibility for
certain Whittaker modules for the fixed point subalgebra of Heisenberg vertex algebra. Further
motivation for this work is to give a vertex algebraic proof of these statements, which can be
applied more generally.

 In  this paper,  we emphasise a new method for proving the irreducibility of orbifold modules    and the
role of Whittaker modules in the proof. In order to avoid  technical details, we study only
non-twisted modules and basic examples of Heisenberg and Weyl vertex algebras. In our forthcoming
papers, we shall study twisted modules and more   examples  of Whittaker modules for $\mathcal
W$--algebras.

\subsection{Irreducibility of orbifolds.}   

{ Let $W$ be an irreducible weak $V$-module and $g$ an
automorphism of $V$ with order $p$. Let $Y_{W}\left(v,z\right)$
be the vertex operator of $v\in V$ operating on $W$. Recall that
$W\circ g$ is defined in \cite{DM} to be the space $W$ with the vertex
operator given by
\[
Y_{ W\circ g}\left(v,z\right)=Y_{W}\left(gv,z\right),\forall v\in V.
\]
 It is clear that $W\circ g$ is also a $V$-module. }
 The following is our first main result  (see Theorem \ref{general} for part (1) and Theorem \ref{reducibility} for part (2)).
 \begin{main} \label{general-intr}    Let $W$ be  an irreducible   weak  $V$--module and $g$ an
automorphism of finite order.
\begin{itemize}
 \item[(1)] Assume that $W \circ g^i \ncong W$ for all $i$. Then $W$ is an
irreducible $V^{\langle g \rangle}$\textendash module.
\item[(2)] Assume that  $W \cong  W \circ g$. Then $W$ is a direct sum  of $p$    irreducible $V^{\langle g \rangle}$--modules.
\end{itemize}
 \end{main}
Let us explain the main new ideas of our proof.  For (1), we construct a graded module $$\mathcal M =
\bigoplus_{i=0} ^{p-1}   W\circ g^i = \bigoplus_{i=0} ^{p-1} \Delta ^{p,i}  (W),$$ compatible with
the action of the automorphism $g$, such that each component is isomorphic to $W$ as $V^{\left\langle g\right\rangle }$--module.
Then we take any non-trivial submodule $S$ of $W$ and identify it with  a  submodule of
$\Delta^{p,0} (W)$. It is then sufficient to prove the following claim:

\begin{enumerate}
\item[(1.1)]   For each $w \ne 0$,  a vector of the form $(w, \dotsc, w) \in \mathcal M$ is cyclic
in $\mathcal M$.
\end{enumerate}
The advantages of our approach are the fact that we do not need Zhu's algebra and the fact that this approach can be applied for non-weight modules.
In Lemma \ref{1.4}, we  prove relation (1.1) for arbitrary weak module by using the Lie algebra $\mathfrak g(V)$ associated to $V$ and its universal enveloping associative algebra. It turns out that (1.1) is just a consequence of a similar statement for associative algebras (cf. Lemma \ref{1.2}).

For proof of the part (2) (cf. Theorem \ref{reducibility}),  we  slightly  modify the methods  of \cite{DM} and \cite{DY}  by applying a general version of Schur's Lemma on the   action of the group $G= {\Z}_p$ on $W$. 

\subsection{Role of the Whittaker modules in the paper}
Although Theorem \ref{general-intr} holds for arbitrary weak $V$--modules, it is not easy to construct examples of modules satisfying the conditions of the theorem.
 It turns out that these conditions can be checked  for a large class of Whittaker modules for certain infinite-dimensional Lie algebras. We use concepts of Whittaker categories which appear in the paper \cite{BM} (see also \cite{MZ}). Since any weak module for a vertex algebra is automatically a module for an infinite-dimensional Lie algebra, such an approach gives a framework for studying many examples. We just need to assume that each module $W\circ g^{i}$ belongs to a different Whittaker block. This means that each module $W \circ g^i $ has a different Whittaker function.
The following is our  second main result (see Theorem \ref{general-whittaker}) which gives most new applications of our construction.

\begin{main} \label{general-intr-whittaker}
Let $W$ be an irreducible weak $V$--module such that all $W_{i}=   W\circ g^{i}$
are Whittaker modules whose Whittaker functions $\lambda^{(i)}=\mathfrak{n}\rightarrow{\C}$ are mutually distinct.
Then $W$ is an irreducible weak  $V^{\left\langle g\right\rangle }$--module.
\end{main}


\subsection{Examples}

We construct a family of Whittaker modules for Heisenberg and Weyl vertex algebra, and apply our new result to prove irreducibility of orbifold subalgebras. In particular, we show that in these cases, standard (= universal) Whittaker modules are irreducible.

In the case of Heisenberg vertex algebra, we use the new method and present an alternative proof of the ${\Z}_2$--orbifolds of Heisenberg vertex algebra \cite{HY}.  

In the case of Weyl vertex algebra $M$, we construct a family of Whittaker modules $M_1 (\bm{\lambda}, \bm{\mu})$ where  $ (\bm{\lambda}, \bm{\mu}) \in {\C} ^n \times {\C} ^n$. We prove:

\begin{main}[see Theorem \ref{weyl-irreducible}] Assume that $\Lambda = (\bm{\lambda}, \bm{\mu})  \ne 0$. Then $M_1(\bm{\lambda}, \bm{\mu})  $ is an irreducible weak  module for the orbifold subalgebra $M^{\Z_p}$, for each $ p \ge 1$.
\end{main}

\vskip 5mm

 {\bf Acknowledgments:}

D. A. wants to thank V. Mazorchuk for  valuable discussions.
This work was done in part during the authors'  stay at the Research Institute for Mathematical
Sciences (RIMS), Kyoto. We would like to thank the organizers of the program `` VOAs and
symmetries '' for the invitation and excellent research conditions. D.A.  and V. P. are  partially
supported   by the
QuantiXLie Centre of Excellence, a project coffinanced
by the Croatian Government and European Union
through the European Regional Development Fund - the
Competitiveness and Cohesion Operational Programme
(KK.01.1.1.01.0004). C.L. is partially supported by a research grant AS-IA-107-M02 of Academia Sinica
and MoST grant  107-2115-M-001 -003 -MY3  of Taiwan.

\section{Preliminaries}

\begin{definition}
A vertex operator algebra $(V, Y, \mathbf{1}, \omega)$ is a $\mathbb{Z}$-graded vector space ${V = \amalg_{n \in \mathbb{Z}}V_{(n)}}$   such that
  $\mbox{wt}(v) = n$ for  $v \in V_{(n)}$,
\[\dim V_{(n)}< \infty,\ \text{for}\ n \in \mathbb{Z},\]
\[\text{and } V_{(n)} = 0\ \text{for $n$ sufficiently small},\]
equipped with a linear map $V \otimes V \to V[[z, z^{-1}]]$, or equivalently,
\[
\begin{split}
V & \to (End V)[[z, z^{-1}]] \\
v & \mapsto Y(v, z) = \sum_{n\in\mathbb{Z}}v_nz^{-n-1}\ (\text{where }v_n \in End V),
\end{split}
\]
$Y(v,z)$ denoting the vertex operator associated with v, and equipped also with two distinguished homogenous vectors $\mathbf{1}\in V_{(0)}$ (the vacuum) and $\omega \in V_{(2)}$.
The following conditions are assumed for $u,v \in V$:

\begin{itemize}
	\item $u_nv = 0$ for $n$ sufficiently large (the lower truncation condition),
	
	\item $Y(\mathbf{1},z) = \mbox{Id}$,
	
	\item $Y(v,z) \mathbf{1} \in V[[z]]$ and $\lim_{z \to 0} Y(v,z)\mathbf{1} = v$ (creation property),
\end{itemize}
and the Jacobi identity holds
 \[z_0^{-1}\delta\left(\frac{z_1-z_2}{z_0}\right)Y(u, z_1)Y(v, z_2) - z_0^{-1}\delta\left(\frac{z_2-z_1}{-z_0}\right)Y(v, z_2)Y(u, z_1)\]
 \[= z_2^{-1}\delta\left(\frac{z_1-z_0}{z_2}\right)Y(Y(u, z_0)v, z_2).\]

Also, the Virasoro algebra relations hold (acting on $V$):
\[[L(m), L(n)] = (m-n)L(m+n) + \frac{1}{12}(m^3-m)\delta_{n+m,0}(\mbox{rk}  \  V)\mathbf{1},\]
for $m, n \in \mathbb{Z}$, where
\[L(n) = \omega_{n+1} \text{ for } n\in \mathbb{Z} \text{ i.e., } Y(\omega, z) = \sum_{n \in \mathbb{Z}} L(n)z^{-n-2}\]
and
\[\text{rk } V \in \mathbb{C},\]
\[L(0)v = nv = (\text{wt }v)v \text{ for } n \in \mathbb{Z} \text{ and } v \in V_{(n)},\]
\[\frac{d}{dz}Y(v,z) = Y(L(-1)v,z).\]

\end{definition}

We say that $g \in Aut_\C (V)$ is an automorphism of a vertex operator algebra $V$ if
\begin{itemize}
\item $g (a_n b) = g(a) _n g(b) $ for all $a, b \in V$, $n \in {\Z}$.

\item $g(\omega) = \omega$.

\end{itemize}

For any group  $G$ of   automorphisms of $V$,  we have the orbifold vertex algebra $ V^{G} = \{ v \in V
\ \vert \ g(v) = v,  \ g \in G\}$, which   is a vertex subalgebra of $V$.
If  $G =\langle  g \rangle$ is cyclic, we write $V^{\left\langle g\right\rangle }$ for $V^G$.


We shall now recall the notions of weak modules and ordinary modules for $V$.

 \begin{definition} A weak $V$--module  is a pair   $(W, Y_W)$ where
$W$ is a complex vector space,  and  $Y_W ( \cdot, z)$ is a linear map
\[ 
\begin{split}
Y_W:\   V &\rightarrow \mbox{End} (W) [[z, z ^{-1}]], \\
a &\mapsto Y_W(a,z) = \sum_{ n \in {\Z}} a_n z ^{-n-1}, 
\end{split}
\]
which satisfies the following conditions for  $a, b \in V$ and $v \in W$:

\begin{itemize}
\item  $a_n v = 0$ for  $n$ sufficiently large.
\item $Y_W({\bf 1},z)=I_W$.

\item The following Jacobi identity holds:
\begin{equation*}
\begin{split}
z_{0}^{-1}\delta\left(\frac{z_{1}-z_{2}}{z_{0}}\right)Y_W(a,z_{1})Y_W(b,z_{2}) -
z_{0}^{-1}\delta\left(\frac{z_{2}-z_{1}}{-z_{0}}\right)
Y_W(b,z_{2})Y_W(a,z_{1})\\
=z_{2}^{-1}\delta\left(\frac{z_{1}-z_{0}}{z_{2}}\right)Y_W(Y(a,z_{0})b,z_{2}).\hspace*{2cm}
\end{split}
\end{equation*}

\end{itemize}
\end{definition}

Let $L(z) = Y(\omega, z) = \sum_{n \in {\Z}} L(n) z ^{-n-2}$. Note that every weak $V$--module is a module for the Virasoro algebra generated by  $L(n)$, $n \in {\Z}$.

\begin{definition} A weak $V$--module $(W, Y_W)$ is called an ordinary $V$--module if the following conditions hold:
\begin{itemize}
\item  The $L(-1)$-derivative property: for any $a \in V$, $$Y_W (L(-1) a, z) = \frac{d}{dz} Y_W(a,z). $$
\item  The grading property: $$W= \oplus_{\alpha  \in {\C} } W(\alpha), \quad W(\alpha) = \{ v \in W \ \vert \  L(0) v = \alpha v\}$$
\item[] such that for every $\alpha$,
   $\dim W(\alpha) <\infty $ and  $W(\alpha + n) = 0$ for sufficiently negative $n \in {\Z}$.
\end{itemize}

\end{definition}

The following result was  proved in \cite{DM}.

\begin{theorem}\cite[Theorem 6.1] {DM} Assume that $(W, Y_W)$ is an    irreducible    {\bf  ordinary} module for the vertex operator algebra $V$. Assume that $g$ is an automorphism of $V$ of prime order $p$ such that $ W\circ g \ncong W$.  Then $W$ is an irreducible module for the orbifold subalgebra
$V^{\left\langle g\right\rangle }$.
\end{theorem}

The goal of this paper is to extend this result for irreducible {\bf weak} modules for vertex operator
algebras.

\vskip 5mm

\section{ On cyclic vectors in a  direct sum of irreducible weak modules}
In this section, we prove one basic, but important technical result on cyclic vectors in a direct
sum of non-isomorphic weak modules for a vertex operator algebra. It turns out that the result can
be  proved much more easily in the context of associative algebras.

First we include the following result for associative algebras:
\begin{lemma}  \label{1.2} Let $\mathcal A$ be an associative algebra with unity.
Assume that $L_i,$  $i=1, \dots, t,$  are  non-isomorphic  irreducible $\mathcal A$-modules and $
\mathcal L = \bigoplus_{i=1} ^t L_i$.  Then  for each $w_i \ne 0$, $w_i \in L_i$,   a vector of the
form $(w_1, w_2, \dots, w_t )$ is cyclic in $\mathcal L$.

\end{lemma}
\begin{proof}
Let $\mathcal U = \mathcal A . (w_1, w_2, \dots, w_t )$ be the $\mathcal{A}$-module generated by $
(w_1, w_2, \dots, w_t )$.
Let $J_i  = \mbox{Ann} (w_i) =\{a\in \mathcal A\mid a.w_i=0\}$ for $1\le i \le t$.   Then $J_i$ is
a left ideal in $\mathcal A$ and $\mathcal A / J_i \cong L_i$. Since   $L_i$'s are irreducible
$\mathcal A$--modules which are mutually  non-isomorphic,  we conclude:
\begin{itemize}
\item ideals $J_i$, $i=1, \dots, t$ ,  are maximal left ideals,
\item $J_i \ne J_j$ for $i \ne j$,
\item
 $ \mathcal  A / \bigcap _{i=1} ^ t J_i \cong \mathcal L.$

\end{itemize}
 Note that in the last conclusion we use the fact that $J_i$'s are maximal left ideals of $
\mathcal A$ and apply Chinese Remainder Theorem.  This implies that there is an
element
$$ u _i \in \bigcap_{1 \le j \le t, \ j \ne i } J_j,  \ u_i \notin J_i. $$
 Then one can construct  the  vector
$$ u_i (w_1, \dots, w_i, \dots , w_t ) =  (0,  \dots 0, u_i w_i, 0, \dots , 0 ),$$  which belongs to $L_i$, so $L_i \subset \mathcal U$ for all $i$. Therefore
$\mathcal L   \subset   \mathcal U$, which implies that  $\mathcal L = \mathcal U$.
 \end{proof}

 We want to show the analogous result for weak modules for a vertex operator algebra $V$.  For this
purpose, we  use  the Lie algebra  $\mathfrak g(V)$ associated to the vertex operator algebra
$V$  (cf. \cite{Bor}, \cite{DLM-1997}).

The Lie algebra $\mathfrak g(V)$ is realized on the vector space $$\mathfrak g(V) = \frac{ V \otimes {\C}[t,t^{-1}] }{   ( L(-1) \otimes 1 + 1 \otimes \frac{d}{dt} ). V \otimes {\C}[t,t^{-1}] },$$
where  the commutator is given by
$$  [ a\otimes t^n, b \otimes t^m ] = \sum_{i=0} ^{\infty} {n \choose i} (a_i b) \otimes t ^{n+ m -i}.$$

Then  by  \cite[Lemma 5.1]{DLM-1997} we have:

\begin{lemma} \label{lie-verteks} Let $V$ be a vertex operator algebra. We have:
 \begin{itemize}
\item Every weak $V$--module $W$  is a $\mathfrak g(V)$--module with the action
$$ v \otimes t^n \mapsto v_n \quad (v \in V, \ n \in {\Z}).$$

\item If $W$ is an  irreducible weak $V$--module, then $W$ is also an irreducible $\mathfrak g(V)$--module.
\end{itemize}
\end{lemma}

 \begin{lemma}  \label{1.4}
Assume that $L_i,$  $i=1, \dots, t,$  are   non-isomorphic  irreducible weak $V$--modules and  $ \mathcal L = \bigoplus_{i=1} ^t L_i$.  Then for each $w_i \ne 0$, $w_i \in L_i$,  a  vector of the form $(w_1, w_2, \dots, w_t )$ is cyclic in $\mathcal L$.

\end{lemma}
\begin{proof}
 By Lemma \ref{lie-verteks}, we have that  $L_i$, $i=1, \dots, t$  are irreducible modules  for the associative algebra $\mathcal A = U( \mathfrak g(V))$. Then the  assertion follows by applying Lemma \ref{1.2}.\end{proof}

  \section{Main result: order $2$ case}
We shall first consider the case of automorphisms of order two.

Let $\theta$ be an order two automorphism of $V$.
 Let
\[
V^{+}=\{v\in V\ \vert\ \theta(v)=v\},\quad V^{-}=\{v\in V\ \vert\ \theta(v)=-v\}.
\]
Then $V^{+}$ is a vertex subalgebra of $V$ and $V^{-}$ is a $V^{+}$-module.

\begin{theorem} \label{order 2 criterion}  Let $V$ be a  vertex operator  algebra and  $W$ be an irreducible weak $V$\textendash module
such that $ W_{\theta} =  W \circ \theta \ncong W$. Then $W$ is an irreducible  weak $V^{+}$\textendash module. \end{theorem}

\begin{proof} Consider a $V$\textendash module $\mathcal{M}=W\oplus W_{\theta}$.
Define now the map
\[
\Delta^{\pm}:W\rightarrow\mathcal{M},\quad w\mapsto(w,\pm w).
\]
Let
$$ \Delta^{\pm}(W) = \{ (w, \pm w) \ \vert  \ w \in W\}. $$
Then we have
\[
\mathcal{M}=\Delta^{+}(W)\bigoplus\Delta^{-}(W).
\]

Moreover, $\Delta^{\pm}$ are $V^{+}$\textendash homomorphisms.
Next we notice that
\begin{align*}
 & V^{+}.\Delta^{+}\left(W\right)\\
 & =\text{Span}_{\mathbb{C}}\left\{ \left(v_{n}w,\theta\left(v\right)_{n}w\right),v\in V^{+},w\in W\right\} \\
 & =\text{Span}_{\mathbb{C}}\left\{ \left(\left(v_{n}w,v_{n}w\right),v\in V^{+},w\in W\right)\right\} \\
 & =\Delta^{+}\left(W\right)
\end{align*}
and
\begin{align*}
 & V^{-}.\Delta^{+}\left(W\right)\\
 & =\text{Span}_{\mathbb{C}}\left\{ \left(v_{n}w,\theta\left(v\right)_{n}w\right),v\in V^{-},w\in W\right\} \\
 & =\text{Span}_{\mathbb{C}}\left\{ \left(\left(v_{n}w,-v_{n}w\right),v\in V^{-},w\in W\right)\right\} \\
 & =\Delta^{-}\left(W\right)
\end{align*}
Assume that $W$ is not irreducible $V^{+}$\textendash module. Then
there is a $V^{+}$\textendash submodule $0\ne S\subsetneqq W$. In
particular, $0\ne\Delta^{+}(S)\subsetneqq\Delta^{+}(W)$. But Lemma
\ref{1.4} implies that $V.\Delta^{+}(S)=\mathcal{M}$.
Since
\[
V^{\pm}.\Delta^{+}(S)\subset\Delta^{\pm}(W),
\]
we must have that $V^{+}.\Delta^{+}(S)=\Delta^{+}(W)$ which is a
contradiction. The proof follows. \end{proof}

\section{General case}

Assume that $g$ is an automorphism of  arbitrary (not necessarily prime) order $p$. Then
\begin{eqnarray}
V=V^{0}\oplus V^{1}\oplus\cdots\oplus V^{p-1} \label{decomp-p-general}
\end{eqnarray}
where
\[
V^{i}=\{v\in V\ \vert\ gv=\zeta^{i}v\}
\]
and $\zeta$ is a primitive $p$\textendash th root of unity.

 Let $W$ be a  weak $V$--module.
Let
\[
\mathcal{M}=W_{0}\oplus W_{1}\oplus\cdots\oplus W_{p-1},
\]
where $W_i =W\circ g^i, i=0,1,\cdots, p-1$.
Let $\Delta^{(p,i)}$ be the $V^{0}$\textendash homomorphism defined
by
\[
w\mapsto(w,(\zeta^{i})w,\cdots,(\zeta^{i})^{p-1}w).
\]

\begin{lemma} We have:
\begin{itemize}
\item[(1)] $\mathcal{M}=\bigoplus_{i=0}^{p-1}\Delta^{(p,i)}(W).$
\item[(2)] $V^{j}.\Delta^{(p,i)}(W)\subset\Delta^{(p,i+j)}(W).$
\end{itemize}
\begin{proof} The proof of (1) is easy. Let us prove (2).

Take $v\in V^{j}$. For $w'\in W_{r}$ we have
\[
v_{n}w'=\zeta^{rj}v_{n}w',
\]
which implies $$v_n (w, (\zeta^i ) w, \cdots, (\zeta^i
)^{p-1} w)
 =  ( v_n w, \zeta^{ i + j }v_n w, \cdots,(\zeta^{i+j}
)^{p-1} v_n w  ) \in \Delta^{(p,
i +j )} (W). $$The Lemma holds. \end{proof} \end{lemma}

\begin{lemma} \label{general-criterion}   Let $W$ be  an irreducible weak $V$--module and $\mathcal M$  be as above. Assume   that  for every $w \in W$, $w \ne 0$,
 $$   (w, \cdots, w)   \quad \mbox{is cyclic in}\ \ \mathcal M  $$
Then $W$ is an irreducible weak $V^{0}$--module. \end{lemma}

\begin{proof}  Assume that $W$
is not a simple $V^{0}$--module. Then there is a $V^{0}$--submodule
$0\ne S\subsetneqq W$. In particular, $$ 0 \ne \Delta^{(p, 0)}
(S) \subsetneqq \Delta^{(p, 0)} (W). \label{formula-1} $$ Since  each $(w,\cdots,w)$, with $w\ne0$,
is a cyclic vector in $\mathcal{M}$,  we get  $$ \label{cyclic}
\mathcal M = V. \Delta^{(p, 0)} (S). $$ This implies that $V^{0}.\Delta^{(p,0)}(S)=\Delta^{(p,0)}(W)$.
A contradiction. The proof follows. \end{proof}

Lemmas \ref{1.4} and  \ref{general-criterion}   imply our main result.

\begin{theorem} \label{general}    Let $W$ be  an irreducible  weak $V$--module, and $g$ an automorphism
of finite order such that $W \circ g^i \ncong W$ for all $i$. Then $W$ is an irreducible weak
$V^{0}$--module. \end{theorem}

\section{On complete reducibility of certain $V^{\langle g \rangle}$--modules}
In this section we shall use the following very general version of Schur's Lemma. Proof can be found in \cite[Section 4.1.2]{Wall}.

\begin{lemma} \label{schur}
Assume that $W_1$ and $W_2$ are irreducible modules for an associative algebra $A$. Assume that $W_1$ and $W_2$ have countable dimensions over ${\C}$. Then
$ \mbox{dim}  \ Hom _{A} (W_1, W_2)\le 1$  and $ \mbox{dim}  \ Hom _{A} (W_1, W_2) = 1$ if and only if $W_1 \cong W_2$.
\end{lemma}

\begin{lemma} \label{countably}
Assume that $V$ is a vertex operator algebra. Then every irreducible weak $V$-module $W$ is  countably dimensional.
\end{lemma}
\begin{proof}
  Note that the vertex operator algebra $V$ is countably dimensional.  Take any $w \in W$. Then  by \cite[Proposition 6.1]{Li} (see also \cite[Proposition 4.1]{DM}), 
  $$ W= V. w = \mbox{Span}_{\C} \{ u_n w \ \vert \ u \in V, n\in {\Z} \}, $$
  which implies that $W$ is also countably dimensional.

\end{proof}

Assume that $g$ is an automorphism of arbitrary order $p$. Then we have the decomposition (\ref{decomp-p-general}).

\begin{theorem}  \label{reducibility} Assume that $g$ is an automorphism of $V$ of finite order $p$ as above. Assume that $W$ is an irreducible weak $V$--module    such that $W \circ g \cong W$.
Then $W$ is completely reducible weak  $V^0$--module such that 
\begin{itemize}
\item[(1)] $W = \bigoplus_{i=0} ^p  W^i $, $ V^i.W^j \subset W^{i + j  \ \mbox{mod} (p)} $, where    $W^j$, $j=1, \dots, p$  are eigenspaces of certain linear isomorphism $\Phi(g) : W \rightarrow W$.
\item[(2)] Each $W^i$ is an irreducible weak $V^{0}$--module.
\item[(3)]  The modules $W^i$, $i=0, \dots p-1$,  are non-isomorphic as weak $V^{0}$--modules.
\end{itemize}
\end{theorem}
\begin{proof}
By Lemma \ref{countably},  $W$ is countably dimensional. Let $f :  W \rightarrow W \circ g$ be a $V$--isomorphism . Then we have
$$ f(u_n w ) = (g u)_n f(w) \quad \forall u \in V, \ w \in W. $$
Applying $f$ $p$-times, we get
$$ f^p (u_n w)  = (g^p u)_n f^p (w) = u_n f^p (w). $$
Therefore $f^p$ is a $V$--endomorphism. Applying Schur's  Lemma for the associative algebra $A = U(\mathfrak g(V))$, we get
that $f^p = a \mbox{Id}_{W}$, where $a$ is a non-zero constant. By rescaling $f$ one gets  an isomorphism $\Phi(g) :  W \rightarrow W\circ g$ such that $\Phi(g)  ^p = \mbox{Id}$. Next we  consider $\Phi(g)  $ as a linear operator on $W$ with the property
$\Phi(g)  ^p  = \mbox{Id}_{W}$. 

  That means $W$ is a $\langle g\rangle$--module and there is $0 \leq j\leq p-1$ and a vector $0\neq w^j\in W$ such that $ \Phi(g)   (w_j) = \zeta^{j} w_j$. 
Clearly, for any  $x\in V^i. w_j= \mbox{Span}_{\C}\{v_n w_j|v\in V^j, n\in\mathbb Z\}$ and $0 \leq i\leq p-1$, we have   $\Phi(g)   (x) = \zeta^{i+j} x$. Without loss, we may assume there is a $0\neq w\in W$ such that $\Phi(g)   (w)=w$. 
 
Now define
$ W^j = V^j. w= \mbox{Span}_{\C}\{v_n w|v\in V^j, n\in\mathbb Z\}$. Then 
\begin{itemize}
\item  $\Phi(g)   (w_j) = \zeta^{j}  w_j$  for each $w_j \in W^j$.  
\item $\Phi(g)   (u_n w_j ) = \zeta^{j}   g(u)_n  w_j  =  \zeta^{i + j}   u_n  w_j $ for $u \in V^i$, $w_j \in W^j$.  
\end{itemize}
This implies that $W = \bigoplus_{i=0} ^p  W^i $, $ V^i.W^j \subset W^{i + j  \ \mbox{mod} (p)} $ and (1) holds.

Assertion (2)  follows from (1).  

Let  $0 \neq U  \neq  W^j$   be a  proper $V^{0}$--submodule  of   $W^j$  and consider
the $V$--submodule $X= V.U$.
Then   $$ X= V^0.U  + V^1.U + ... +  V^{p-1}.U$$
Since  $U$ is a proper  $V^{0}$--submodule of $W^j$, then $ V^i.U \subset  W^{i+j}$ implies that  $X$ is a proper $V$--submodule of $W$.
This is impossible since $W$ is a simple  $V$--module.
Hence $ W ^j$ is irreducible $V^0$--module for each $j$.

   Proof of assertion (3) is    completely analogous to that of \cite[Theorem 5.1]{DM} and it uses Lemma \ref{schur}. For completeness, we shall include it.
 
Suppose    we have a $V^{0}$--isomorphism   $p: W^ i  \to W^j $, $i\neq j$.
Take a nonzero $ w \in W^i$  and consider the following  $V$--submodule $\mathcal U$   of  $W\oplus W$
$$\mathcal U = V.  (w,  p(w))= \mbox{Span}_{\C} \{ (v_n w, v_n p(w) ) \ \vert \ v \in V \}.  $$ 
 Then   $ W^i  \oplus W^i $ is not in $\mathcal U$ and hence $\mathcal U$  is a proper
submodule of  $W\oplus W$. 
Since $W \oplus W$ is a $U(\frak g(V))$--module of finite length, the Jordan H\" older theorem can be applied.  
Comparing filtrations
$$ (0) \rightarrow W \rightarrow W\oplus W, \quad  (0) \rightarrow \mathcal U \rightarrow W\oplus W,$$
and using simplicity of $W$, we get that  $\mathcal U  \cong   W$ as  $U(\frak g(V))$--modules. This implies that $\mathcal U  \cong   W$ as $V$--modules.

Then both projection maps  from $\mathcal U  \to  W \oplus ( 0 ) $  and $\mathcal U \to ( 0 )  \oplus W$
are  $V$--isomorphisms.
Hence the map   $$ \Phi :   u_n w \mapsto   u_n p(w), \quad (u \in V)  $$ is also an isomorphism. Using Schur's lemma we get $\Phi = a \mbox{Id}$ for $a \in {\C}$,  which implies that
$p(w) = a w \in W^i$. This implies $i=j$.  A contradiction

\end{proof}

\section{Whittaker modules: some structural results}

First we recall some basic notions from \cite{BM}.

\begin{definition} For a Lie algebra $\mathfrak{n},$ define ideals
$\mathfrak{n}_{0}:=\mathfrak{n}$ and $\mathfrak{n}_{i}:=\left[\mathfrak{n}_{i-1},\mathfrak{n}\right],i>0$.
Then we have a sequence of ideals
\[
\mathfrak{n}=\mathfrak{n}_{0}\supset\mathfrak{n}_{1}\supset\mathfrak{n}_{2}\supset\cdots.
\]

We say that $\mathfrak{n}$ is \emph{quasi-nilpotent }if $\cap_{i=0}^{\infty}\mathfrak{n}_{i}=0$.
Obviously, any nilpotent Lie algebra is quasi-nilpotent.

\end{definition}

\begin{definition} Let $\mathfrak{g}$ be a nonzero complex Lie algebra
and let $\mathfrak{n}$ be a subalgebra of $\mathfrak{g}$. If $M$
is a $\mathfrak{g}$--module, then we say that the action of $\mathfrak{n}$
on $M$ is \emph{locally finite} provided that $U\left(\mathfrak{n}\right)v$
is finite dimensional for all $v\in M$. Let ${\mathcal{W}h}(\mathfrak{g},\mathfrak{n})$
denote the  full subcategory of the category $\mathfrak{g}$--Mod of
all $\mathfrak{g}$--modules, which consists of all $\mathcal{\mathfrak{g}}$--modules,
the action of $\mathfrak{n}$ on which is locally finite.

\end{definition}

Let  $V$ be a vertex algebra. Assume that the Lie algebra $\mathfrak L$ is one of the following two types:
\begin{itemize}
\item[(1)] $\mathcal L = \mathfrak g(V)$, or
\item[(2)] $\mathcal L$ is the Lie algebra of modes of generating fields of the vertex algebra $V$.
\end{itemize}

\begin{remark} In many cases it is possible to replace $\mathfrak g(V)$ with much smaller Lie algebra. For example, this happens in the following cases:
\begin{itemize}
\item If $V$ is  the universal affine vertex algebra $V^k(\g)$, then we can take $\mathfrak L = {\hg}$, where $\hg$ is the affine Lie algebra associated to the simple Lie algebra $\g$ (cf. \cite{ALZ}, \cite{A-2019} for studying Whittaker modules in this case).
\item If $V$ is the Heisenberg vertex algebra $M(1)$, we can take  $\mathfrak L = {\hh}$ (cf. Section \ref{Heisenberg} below).
\item If  $V$ is the Weyl vertex algebra, we can take Lie algebra $\mathfrak L$ such that the Weyl algebra $\widehat {\mathcal A} =U(\mathfrak L)$ (cf. Section \ref{Weyl} below).
\end{itemize}

\end{remark}
Note that by our assumptions on the vertex algebra, every weak $V$\textendash module is a module for the Lie algebra $\mathfrak{L}$. We also assume the following:
\begin{itemize}
\item Let $\mathfrak{n}$ be a nilpotent subalgebra of $\mathfrak L $.
\item Let ${\mathcal{W}h}(\mathfrak{L},\mathfrak{n})$   denote  the full
category of $\mathfrak{L}$\textendash modules $W$ such that $\mathfrak{n}$
acts locally finitely on $W$ (cf. \cite{BM}).
\end{itemize}

\begin{definition} Let $W\in{\mathcal{W}h}(\mathfrak{L},\mathfrak{n})$.
A vector $v\in W$ is called a Whittaker vector provided that $\left\langle v\right\rangle $
is an $\mathfrak{n}$-submodule of $W$. \end{definition}

Let  $\lambda : \mathfrak n \rightarrow {\C}$  be a Lie algebra homomorphism which will be called a {\em Whittaker function}. Let $U_{\lambda}= {\C} u_{\lambda}$ be the $1$--dimensional $\mathfrak n$--module such that $$ x u_{\lambda} = \lambda(x) u_{\lambda} \quad (x \in \mathfrak n). $$
Consider the standard (universal)  Whittaker $\mathfrak{L}$--module
$$ M_{\lambda} = U(\mathfrak L) \otimes_{U(\mathfrak n)} U_{\lambda}. $$

\begin{definition} \label{def-2.4} We say that an irreducible $V$\textendash module
$W$ is of Whittaker type $\lambda$ if $W$ is an irreducible quotient of the standard Whittaker module $M_{\lambda}$.
\end{definition}

\begin{lemma}[cf. \cite{BM}] \label{desc}
Assume that $W$ is an irreducible $V$--module of Whittaker type $\lambda$. Then $$W = \{ w  \in W \  \vert \ \forall  x \in {\mathfrak n}, \quad  (x - \lambda (x) ) ^ k w = 0 \ \mbox{for} \ k >> 0 \}. $$
\end{lemma}
\begin{proof} 
Let us first prove that
$$ M_{\lambda} =  \{ w  \in M_{\lambda} \  \vert \ \forall  x \in {\mathfrak n}, \quad  (x - \lambda (x) ) ^ k w = 0 \ \mbox{for} \ k >> 0 \}. $$
Take an arbitrary element $w_1 \in  U(\mathfrak{L})$. Since $\mathfrak n$ is a nilpotent subalgebra of $\mathfrak{L}$ ,  for $x \in {\mathfrak n}$, we have
$$ \mbox{ad}_x ^k (w_1) = 0  \quad \mbox{for} \ k >>0 . $$
Next we notice that
$$ ( x-\lambda(x)) w_1 \otimes u_{\lambda} = [x, w_1] \otimes u_{\lambda}$$
which implies that
$$ ( x-\lambda(x)) ^k  w_1 \otimes u_{\lambda} = \mbox{ad}_x ^k (w_1)  \otimes u_{\lambda} = 0 \quad \mbox{for} \ k >> 0.$$
The claim now follows from the fact that $W$ is a quotient of the universal Whittaker module $M_{\lambda}$.
%
\end{proof}

\begin{lemma} \label{inequivalence =000026 cyclic vector }Assume
that $\lambda,\mu:\mathfrak{n}\rightarrow{\C}$ are Whittaker functions such that $\lambda\ne\mu$.
Assume that $W_{\lambda}$ and $W_{\mu}$ are irreducible Whittaker
modules of types $\lambda$ and $\mu$ respectively. Then

(1) $W_{\lambda}$ and $W_{\mu}$ are inequivalent as $V$\textendash modules.

(2) Let $(w_{1},w_{2})\in W_{\lambda}\oplus W_{\mu}$, $w_{1}\ne0,w_{2}\ne0$.
Then
\[
V.(w_{1},w_{2})=W_{\lambda}\oplus W_{\mu}.
\]
\end{lemma}

\begin{proof} (1) Assume that $f:W_{\lambda}\rightarrow W_{\mu}$ is
an isomorphism. Then $f(w_{\lambda})$ is a non-trivial Whittaker
vector in $W_{\mu}$ such that
$$( x -\lambda(x) )  f(w_{\lambda}) =0, \quad \forall x\in \mathfrak{n}.  $$
Take $x\in \mathfrak{n}$ such that $\lambda(x) \ne \mu(x)$.
The for every $k >0$ we have
$$( x -\mu(x) ) ^k   f(w_{\lambda}) = (x - \lambda(x) + \lambda(x) -\mu(x)) ^k   f(w_{\lambda})  = (\lambda (x)  -\mu (x) ) ^k f(w_{\lambda})  \ne 0. $$
This contradicts with Lemma  \ref{desc}. So (1) holds.

Now let us prove (2).

Claim: There exist $k>0$ and $x\in{\mathfrak{n}}$ such that $$(x-\lambda(x))^{k}w_{1}=0\quad
\mbox{and}\quad (x-\lambda(x))^{k}w_{2}=z'\ne0.$$
Indeed, take $x\in\mathfrak{n}$ such that $\mu(x)-\lambda(x)=A\ne0$.
Let $k_{1},k_{2}$ be the smallest positive integer such that
\[
(x-\lambda(x))^{k_{1}}w_{1}=0\ \mbox{and}\ \ (x-\mu(x))^{k_{2}}w_{2}=0.
\]
Let $k=\max\{k_{1},k_{2}\}$. We have
\begin{align*}
 & (x-\lambda(x))^{k}w_{2}\\
 & =(x-\mu(x)+A)^{k}w_{2}\\
 & =\sum_{p=0}^{k-1}\left(_{p}^{k}\right)\left(x-\mu\left(x\right)\right)^{p}A^{k-p}w_{2}\\
 & =\sum_{p=0}^{k_{2}-1}\left(_{p}^{k}\right)\left(x-\mu\left(x\right)\right)^{p}A^{k-p}w_{2}\\
 & =A^{k}w_{2}+kA^{k-1}\left(x-\mu\left(x\right)\right)w_{2}+\cdots+\left(_{k_{2}-1}^{k}\right)A^{k-k_{2}+1}\left(x-\mu\left(x\right)\right)^{k_{2}-1}w_{2}.
\end{align*}
Note that $w_{2},\left(x-\mu\left(x\right)\right)w_{2},\cdots,\left(x-\mu\left(x\right)\right)^{k_{2}-1}w_{2}$
are linearly independent. Otherwise, there exist $p_{0},p_{1},\cdots,p_{k_{_{2}}-1}$
such that
\[
p_{0}w_{2}+p_{1}\left(x-\mu\left(x\right)\right)w_{2}+\cdots+ p_{k_{2}-1}\left(x-\mu\left(x\right)\right)^{k_{2}-1}w_{2}=0.
\]
Let $I=\left\{ i=0,1,\cdots,k_{2}-1|p_{i}\not=0\right\} $ and $s=\max I.$
Then
\begin{equation}
\left(x-\mu\left(x\right)\right)^{s}w_{2}=\sum_{i=0}^{s-1}\frac{p_{i}}{p_{s}}\left(x-\mu\left(x\right)\right)^{i}w_{2}.\label{eq1}
\end{equation}
 If $\left(x-\mu\left(x\right)\right)^{k_{2}-s}w_{2},$ $\left(x-\mu\left(x\right)\right)^{k_{2}-s+1}w_{2},$
$\cdots,$$\left(x-\mu\left(x\right)\right)^{k_{2}-1}w_{2}$ are linearly
independent, then $p_{i}=0$, $i=0,1,\cdots,s-1.$ By (\ref{eq1}),
we obtain $\left(x-\mu\left(x\right)\right)^{s}w_{2}=0$ where $s<k_{2}$,
which is a contradiction. So$\left(x-\mu\left(x\right)\right)^{k_{2}-s}w_{2},$
$\left(x-\mu\left(x\right)\right)^{k_{2}-s+1}w_{2},$ $\cdots,$ $\left(x-\mu\left(x\right)\right)^{k_{2}-1}w_{2}$
(at most $k_{2}-1$ terms) are linearly dependent. Repeating similar
argument, we can prove that there exists $q<k_{2}$ such that $\left(x-\mu\left(x\right)\right)^{q}w_{2}=0$,
which is a contradiction. Thus we proved $w_{2},$ $\left(x-\mu\left(x\right)\right)w_{2},\cdots,\left(x-\mu\left(x\right)\right)^{k_{2}-1}w_{2}$
are linearly independent and hence
\[
A^{k}w_{2}+kA^{k-1}\left(x-\mu\left(x\right)\right)w_{2}+\cdots+\left(_{k_{2}-1}^{k}\right)A^{k-k_{2}+1}\left(x-\mu\left(x\right)\right)^{k_{2}-1}w_{2}\not=0.
\]
Now we have
\[
(x-\lambda(x))^{k}(w_{1},w_{2})=(0,z'),\quad z'\in W_{\mu},z'\ne0.
\]
Irreducibility of $W_{\mu}$ now gives that $W_{\mu}\subset V.(w_{1},w_{2})$.
Similarly we prove that $W_{\lambda}\subset V.(w_{1},w_{2})$. So
(2) holds. \end{proof}
\begin{remark}
The assertion (2) is a consequence of (1) and Lemmas \ref{1.2} and \ref{1.4}.  So we could omit its
proof. But since Whittaker modules  are   the main  new examples on which we can apply Theorem
\ref{general}, we think that it is important to keep an independent proof which contains explicit
construction of elements in maximal left ideals $J_i=Ann (w_i)$, $i=1,2$. In particular, we have
elements $ (x-\lambda(x))^{k} \in J_1 \setminus J_2$ which correspond to element $u_i$ (for $i=2$)
constructed in Lemma \ref{1.2} by using  abstract arguments.
\end{remark}

We can easily generalize the previous lemma:
\begin{lemma} \label{inequivalence =general case}Assume
that $\lambda_1,\cdots, \lambda_n :\mathfrak{n}\rightarrow{\C}$  are Whittaker functions such that $\lambda_i \ne\lambda_j$ for $i\ne j$.
Assume that $W_{\lambda_i}$, $i =1, \dots, n$  are irreducible Whittaker
modules of types $\lambda_i$. Then
\begin{itemize}
\item[(1)] All  $W_{\lambda_i}$  are  inequivalent as $V$\textendash modules.

\item[(2)] Let $ {\bf w }= (w_{1},w_{2}, \cdots w_n)\in  W_{\lambda_1}\oplus \cdots  \oplus W_{\lambda_n} $, where $0\ne w_{i}\in W_{\lambda_i} $.
Then
\[
V. {\bf w} =W_{\lambda_1}\oplus \cdots  \oplus W_{\lambda_n}.
\]
\end{itemize}
\end{lemma}

\begin{theorem} \label{general-whittaker}
Assume that $W$ is a $V$\textendash module in the Whittaker category
${\mathcal{W}h}(L,\mathfrak{n})$ as before. Assume also that $W_{i}= W\circ g^{i} $
has the Whittaker function $\lambda^{(i)}=\mathfrak{n}\rightarrow{\C}$
and that all $\lambda^{(i)}$ are distinct.  Then $W$ is an irreducible $V^0$--module.
\end{theorem}

\begin{proof}
The proof follows immediately from Lemma \ref{inequivalence =general case} and Theorem \ref{general-criterion}.
\end{proof}

\section{Example: Heisenberg vertex algebra}
\label{Heisenberg}

First we recall the definition of the Heisenberg  Lie algebra $\hat{\mathfrak{h}}$. 
Let
$\mathfrak{h}$ be complex  $\ell$-dimensional vector spaces with respect to the  non-degenerate bilinear form $(\cdot, \cdot)$. Fix
an orthonormal basis $\{h_{1},h_{2},\ldots,h_{\ell}\}$ with respect to  form  $(\cdot, \cdot)$.
 The Heisenberg  Lie algebra $\hat{\mathfrak{h}}$  is defined as
\[
\hat{\mathfrak{h}}=\mathfrak{h}\otimes\mathbb{C}\left[t,t^{-1}\right]\oplus\mathbb{C}K
\]
with the commutator relations
\[
[K,\hat{\mathfrak{h}}]=0,
\quad \text{ and } \quad [a(m),b(n)]=m\delta_{m+n,0}(a,b)K
\]
for $a,b\in\mathfrak{h}$, $m,n\in\mathbb{Z}$ and $a(n)=a\otimes t^{n}$.
We identify $\mathfrak{h}$ with its dual space $\mathfrak{h}^{*}$ by
the form $(\cdot,\cdot)$. Let $\mathbb{C}e^{0}$ be the one-dimensional
module over the Lie algebra $\hat{\mathfrak{h}}$ with action given
by
\[
h(n)e^{0}=0,\quad\forall h\in\mathfrak{h},\,n\ge0;\qquad Ke^{0}=e^{0}.
\]
Define the vector space $M(1)$ by
\begin{equation}
M(1)=U(\hat{\mathfrak{h}})\otimes_{U(\mathfrak{h}\otimes\mathbb{C}[t]\oplus\mathbb{C}K)}\mathbb{C}e^{0}.
\end{equation}
On $M(1)$ define the state-field correspondence by
\begin{equation}
Y(a^{(1)}(n_{1})\cdots a^{(r)}(n_{r})e^{0},z)=a^{(1)}(z)_{n_{1}}\cdots a^{(r)}(z)_{n_{r}}\text{Id}_{M(1)}
\end{equation}
for $a^{(i)}\in\mathfrak{h}$ and $n_{i}\in\mathbb{Z}$. The vacuum
vector is $\boldsymbol{1}=e^{0}$ and the conformal vector is given
by

\[
\omega=\frac{1}{2}\sum_{i=1}^{\ell}h_{i}(-1)^{2}\boldsymbol{1}.
\]
In particular,
\[
Y(\omega,z)=L(z)=\sum_{n\in\mathbb{Z}}L(n) z^{-n-2},\qquad L(n)=\frac{1}{2}\sum_{m\in\mathbb{Z}}{}_{\circ}^{\circ}h_{i}(-m)h_{i}(m+n){}_{\circ}^{\circ}.
\]

Then $\left(M\left(1\right),Y,\boldsymbol{1},\omega\right)$ is a
vertex operator algebra (see \cite{LL} for details).
 Now consider
the order two automorphism $\theta$ of the vector space $M(1)$ given
by
\[
\theta\big(h_{i_{1}}(-n_{1})h_{i_{2}}(-n_{2})\cdots h_{i_{k}}(-n_{k})\boldsymbol{1}\big)=(-1)^{k}h_{i_{1}}(-n_{1})h_{i_{2}}(-n_{2})\cdots h_{i_{k}}(-n_{k})\boldsymbol{1}
\]
for $i_{j}\in\{1,2,\ldots,\ell\}$ for all $j$ and $n_{1}\ge n_{2}\ge\cdots\ge n_{k}>0$.
Let $M(1)^{+}$ be the corresponding subspace of fixed-points with
respect to $\theta$:
\begin{equation}
M(1)^{+}=\big\{ v\in M(1)\mid\theta(v)=v\big\}.
\end{equation}
Then $M\left(1\right)^{+}$ is a vertex operator algebra and its structure
and representation are well studied (cf. \cite{DN1,A1,DN3,ADL}).

Let $\boldsymbol{\lambda}=(\lambda_{0},\lambda_{1},\ldots,\lambda_{r},0,\cdots)$
be a sequence of elements of $\mathfrak{h}$ with at least one nonzero
entry, and $\lambda_{n}=0$ for $n\gg0$. Let $\mathbb{C}e^{\boldsymbol{\lambda}}$
be the one-dimensional module over the Lie algebra $\mathfrak{h}\otimes\mathbb{C}[t]\oplus\mathbb{C}K$
with action given by
\[
h(n)e^{\boldsymbol{\lambda}}=(h,\lambda_{n})e^{\boldsymbol{\lambda}},\;h\in\mathfrak{h},\,n\ge0;\quad Ke^{\boldsymbol{\lambda}}=e^{\boldsymbol{\lambda}}.
\]
Consider the corresponding induced $U(\widehat{\mathfrak{h}})$-module
\[
M(1,\boldsymbol{\lambda})=U(\hat{\mathfrak{h}})\otimes_{U(\mathfrak{h}\otimes\mathbb{C}[t]\oplus\mathbb{C}K)}\mathbb{C}{e^{\boldsymbol{\lambda}}}.
\]

Let $\mathfrak{L}=\hat{\mathfrak{h}}$ and $\mathfrak{n}=\mathfrak{h}\otimes\mathbb{C}\left[t\right]\oplus\mathbb{C}K$,
then it is clear that $\mathfrak{n}$ is a nilpotent subalgebra of
$\mathfrak{L}$ and hence $M(1,\boldsymbol{\lambda})$ is the standard
(universal) Whittaker $\mathfrak{L}$-module $U\left(\mathfrak{L}\right)\otimes_{U\left(\mathfrak{n}\right)}$$\mathbb{C}e^{\boldsymbol{\lambda}}$.
Define the Whittaker function $\Lambda:\mathfrak{n}\to\mathbb{C}$
by
\[
\Lambda\left(h\left(n\right)\right)=\left(h,\lambda_{n}\right),n=0,1,\cdots,r;\ \Lambda\left(h\left(k\right)\right)=0,\forall k>r.
\]
Then we see that $\mathbb{C}e^{\boldsymbol{\lambda}}$ is a 1-dimensional
$\mathfrak{n}$-module such that $xe^{\boldsymbol{\text{\ensuremath{\lambda}}}}=\Lambda\left(x\right)e^{\boldsymbol{\lambda}}$
for any $x\in\mathfrak{n}$. By Definition \ref{def-2.4}, $M(1,\boldsymbol{\lambda})\in\mathcal{W}h$$\left(\mathfrak{L},\mathfrak{n}\right)$
is an irreducible \emph{Whittaker module }for $M\left(1\right)$ of
Whittaker type \emph{$\boldsymbol{\Lambda}$. }

Now we see that $ M\left(1,\boldsymbol{\lambda}\right)\circ\theta$
is a Whittaker module for $M\left(1\right)$ with type $-\Lambda$.
By Theorem \ref{general-whittaker}, $M\left(1,\boldsymbol{\lambda}\right)$
is irreducible as Whittaker module for $M\left(1\right)^{+}$. This
gives another proof of Theorem 6.1 in \cite{HY}.

\subsection{On cyclic orbifolds of $M(1)$}
The orbifolds of $M(1)$ were studied by A. Linshaw in \cite{L} using invariant theory. We can now prove irreducibility of certain Whittaker modules for $M(1)$ for   cyclic orbifolds.

Let $O(\ell)$ be orthogonal group. It acts naturally on the vector space ${\h}$ by preserving form $(\cdot, \cdot)$:
$$ (g h, g h') = (h, h') \quad \forall h, h' \in {\h}, \ g \in O(\ell). $$

 The action of $O(\ell)$ on ${\h}$ can be uniquely extended to the action on the vertex operator algebra $M(1)$. Moreover, $O(\ell)$ is the full automorphism group of $M(1)$. 
 Let $ \bm{\lambda} : \mathfrak n \rightarrow {\C}$ be a Whittaker function. 
The action of $O(\ell) $ on $M(1,\bm{\lambda})$ is given by
\[
M(1,\bm{\lambda})\circ g=M(1,\bm{\lambda}\circ g),
\]
$$
  ( \bm{\lambda}\circ g ) (h(n))  =    \bm{\lambda}   (g h)(n) , \quad g \in  O(\ell), \ h\in {\h}, \ n \ge 0. 
$$
It is  important  to consider orbifolds of certain finite subgroups of $O(\ell)$.  A particularly interesting subgroup of $O(\ell)$  is the symmetric group $S_{\ell}$ of $\ell$ letters. A detailed example of
orbifold $M(1) ^{S_3}$, in the case of rank three Heisenberg algebra,  were presented in a recent
paper by A. Milas, M. Penn and  H. Shao \cite{MPS}.

Let  $h_1, \dots, h_{\ell}$ be  the basis of $\mathfrak h$  as above.   Let $ \bm{\lambda} : \mathfrak n \rightarrow {\C}$ be a Whittaker function. Then
 $$ \bm{\lambda} = (\lambda^ 1, \dots, \lambda ^{\ell}), \quad \lambda ^i = ( \lambda (h_i(0)),   \lambda (h_i(1)), \cdots ).$$
The action of $S_{\ell}$ on $M(1,\bm{\lambda})$ is given by
\[
M(1,\bm{\lambda})\circ g=M(1,\bm{\lambda}\circ g),
\]
\[
\bm{\lambda}\circ g=(\lambda^{g(1)},\dots,\lambda^{g(\ell)})\quad\forall g\in S_{\ell}.
\]

\begin{proposition}

 \begin{itemize}
 
 \item[(1)]  Assume that $g\in O(\ell) $ is of finite order such that $\bm{\lambda}\circ g^{i}\ne\bm{\lambda}$ for all $i$. Then $M(1,\boldsymbol{\lambda})$ is an irreducible $M(1)^{\left\langle g\right\rangle}$--module.
\item[(2)]   Assume that $\bm{\lambda}\circ\sigma\ne\bm{\lambda}$ for any
$2$\textendash cycle $\sigma\in S_{\ell}$. Then $M(1,\boldsymbol{\lambda})$ is an irreducible $M(1)^{\left\langle g\right\rangle}$--module for any $g\in S_{\ell}$.
\end{itemize}
\end{proposition}

\begin{proof} The proof of assertion (1) follows easily by using
Theorem \ref{general-criterion}. Since for a $2$-cycle $\sigma$,  we
have
\[
\bm{\lambda}\circ\sigma\ne\bm{\lambda},\ \forall\sigma\iff\forall i,j,\quad1\le i<j\le\ell,\quad\lambda_{i}\ne\lambda_{j}
\]
\[
\iff(\lambda^{h(1)},\dots,\lambda^{h(\ell)})\ne(\lambda^{1},\dots,\lambda^{\ell}),\quad\forall h\in S_{\ell}.
\]
Therefore for any $g \in S_{\ell}$ we have $\bm{\lambda}\circ g^i  \ne\bm{\lambda}$.
Now assertion (2) follows directly from (1). \end{proof}

We also have the following  conjecture.


\begin{conjecture} Assume that $\bm{\lambda}\circ\sigma\ne\bm{\lambda}$
for any $2$\textendash cycle $\sigma\in S_{\ell}$. Then $M(1,\boldsymbol{\lambda})$
is an irreducible $M(1)^{S_{\ell}}$\textendash module. \end{conjecture}

 The proof of the conjecture requires certain extension of methods used in the paper. We plan to study the proof of this conjecture in our forthcoming papers.


\section{Example: Weyl vertex algebra}
\label{Weyl}

The Weyl algebra $\widehat{\mathcal{A}}$ is  the associative algebra with generators $a(n), a^{*}(n), n \in \mathbb{Z}$ and relations
\[[a(n), a^{*} (m)] = \delta_{n+m,0}, \quad [a(n), a(m)] = [a ^{*} (m), a ^* (n) ] = 0, \quad\ n,m\in \mathbb{Z}. \]

Let $M$ be the simple {\it Weyl module} generated by the cyclic vector ${\bf 1}$ such that
$$ a(n) {\bf 1} = a  ^* (n+1) {\bf 1} = 0 \quad (n \ge 0). $$
As a vector space, $$ M \cong {\C}[a(-n), a^*(-m) \ \vert \ n > 0, \ m \ge 0 ]. $$

 There is a unique vertex algebra $(M, Y, {\bf 1})$  (cf. \cite{F,FB,KR}) where
the  vertex operator is given by $$ Y: M \rightarrow \mbox{End}(M) [[z, z ^{-1}]] $$
such that
$$ Y (a(-1) {\bf 1}, z) = a(z), \quad Y(a^* (0) {\bf 1}, z) = a ^* (z),$$
$$ a(z)   = \sum_{n \in {\Z} } a(n) z^{-n-1}, \ \ a^{*}(z) =  \sum_{n \in {\Z} } a^{*}(n)
z^{-n}. $$

We choose the following conformal vector of central charge $c=-1$ (cf. \cite{KR}): $$ \omega = \frac{1}{2} (a(-1) a^{*}(-1) - a(-2) a^{*} (0)) {\bf 1}.$$
Then $(M, Y, {\bf 1}, \omega)$  has the structure of a $\frac{1}{2} {\Z}_{\ge 0}$--graded vertex operator algebra. We can define weak and ordinary modules for  $(M, Y, {\bf 1}, \omega)$ as in the case of ${\Z}$--graded vertex operator algebras.

We define the Whittaker module for $\widehat{\mathcal{A}}$ to be the quotient  \[M_1(\bm{\lambda}, \bm{\mu}) = \sfrac{\widehat{\mathcal{A}}}{I},\]
 where $\bm{\lambda} = (\lambda_0, \dotsc, \lambda_n)$, $\bm{\mu} = (\mu_1, \dotsc, \mu_n)$ and $I$ is the left ideal \[I = \big\langle a(0)-\lambda_0, \dotsc, a(n) - \lambda_n, a^{*}(1)-\mu_1, \dotsc, a^{*}(n) - \mu_n, a(n+1), \dotsc, a^{*}(n+1), \dotsc \big\rangle.\]

\begin{proposition} We have:
\item[(1)]  $M_1(  \bm{ \lambda}, \bm{ \mu}) $ is an irreducible $\widehat{\mathcal{A}}$--module.
\item[(2)] $M_1(  \bm{ \lambda}, \bm{ \mu}) $ is an irreducible weak module for the Weyl vertex operator algebra $M$.
\end{proposition}
\begin{proof}
It is straightforward to check that the ideal $I$ defined above is a maximal left ideal in $\mathcal A$ (see also \cite{BO}, \cite{FGM})  and therefore  the quotient  $M_1(  \bm{ \lambda}, \bm{ \mu}) =\mathcal A / I$ is an simple $\widehat{\mathcal{A}}$--module. Since by construction,  $M_1(  \bm{ \lambda}, \bm{ \mu})$  is a restricted  $\widehat{\mathcal{A}}$--module, it is an irreducible $M$--module.
\end{proof}

Let $w = 1 + I \in M_1(\bm{\lambda}, \bm{\mu})  $. Then $w$ is a cyclic vector and
$$  a(0) w = \lambda_0 w, \dots, a(n) w = \lambda_n w,  a^*(1) w = \mu_1 w, \dots , a ^* (n) w =  \mu_n w   $$
and $a ^* (k) w = a(k) w = 0$ for $k >n$.

Now we want to identify $M_1(\bm{\lambda}, \bm{\mu})$ as a Whittaker module for certain Whittaker pair. Let $\mathfrak L$ be the Lie algebra with generators $a(n), a^{*}(n), K$, $n \in {\Z}$  such that $K$ is central   and
\[[a(n), a^{*} (m)] = \delta_{n+m,0} K, \quad [a(n), a(m)] = [a ^{*} (m), a ^* (n) ] = 0, \quad\ n,m\in \mathbb{Z}. \]
Then $M_1(\lambda, \mu)$ is an irreducible  $\mathfrak L$--module of level $1$ (i.e., $K$ acts as the multiplication with $1$).

Let $\mathfrak n$ be the subalgebra of $\mathfrak L$ generated by $a (n), a^*(n+1)$ for $n \ge 0$. Then $\mathfrak n$ is commutative, and therefore a nilpotent subalgebra of $\mathfrak L$.

Define the  Whittaker function $\Lambda =(\bm{\lambda}, \bm{\mu})  : \mathfrak n \rightarrow {\C}$ by
$$ \Lambda (a(0)) = \lambda_0, \dots, \Lambda(a(n)) = \lambda_n, \Lambda(a(k)) = 0 \quad (k >n), $$
$$ \Lambda (a^* (1)) = \mu_1, \dots, \Lambda(a^* (n)) = \mu_n, \Lambda(a^ * (k)) = 0 \quad (k >n). $$

\begin{proposition}
$M_1(\bm{\lambda}, \bm{\mu})  $ is a standard Whittaker module  of level $1$ for the Whittaker pair  $(\mathfrak L, \mathfrak n)$ with Whittaker function  $\Lambda = (\lambda, \mu)$.
\end{proposition}

Let $\zeta_p = e^{2 \pi i / p}$ be $p$-th root of unity. Let $g_p$ be the automorphism of  the vertex operator  algebra $M$ which is uniquely determined by the following automorphism  of the Weyl algebra  $\widehat{\mathcal{A}}$:
$$ a(n) \mapsto \zeta_p a(n), \quad a^* (n) \mapsto \zeta^{-1} _p a^* (n) \quad (n \in {\Z}). $$
Then $g_p$ is the automorphism of $M$ of order $p$.


\begin{theorem} \label{weyl-irreducible} Assume that $\Lambda=(\bm{\lambda},\bm{\mu})\ne0$.
Then $M_{1}(\bm{\lambda},\bm{\mu})$ is an irreducible weak module for
the orbifold subalgebra $M^{\Z_{p}}=M^{\left\langle g_{p}\right\rangle}$ for each $p\ge1$.
\end{theorem} \begin{proof} First we notice that $M_{1}(\bm{\lambda},\bm{\mu})\circ g^{i}=M_{1}(\zeta_{p}^{i}\bm{\lambda},\zeta_{p}^{-i}\bm{\mu})$
and therefore modules $M_{1}(\bm{\lambda},\bm{\mu})\circ g^{i}$ have
different Whittaker functions for $i=0,\dots,p-1$. Now assertion
follows from Theorem \ref{general-whittaker}. \end{proof}

\subsection{An application to affine VOA}

Let $\mathfrak{g}$ be a simple Lie algebra and let $V^{k}(\mathfrak{g})$ be its universal affine vertex algebra of level $k$. Let $V_{k}(\mathfrak{g})$ be its simple quotient.
The following result is well-known:
\begin{lemma} \label{ired-11}
If $W$ is an irreducible weak  $V_{k}(\mathfrak{g})$-module, then $M$ is an irreducible module for the affine Lie algebra  $\hat {\mathfrak{g}}$-module of level $k$.
\end{lemma}

Next we show how the Theorem \ref{weyl-irreducible} gives a construction of new irreducible modules for affine Lie algebra $\widehat{sl(2)}$ associated to $sl(2)$.
In the case $p=2$, ${\Z}_2$--orbifold $M^{\Z_2}$  is isomorphic to a simple affine VOA $V_{-1/2}(sl(2))$   (cf. \cite{FF} and also \cite[Section 6]{AMPP}) associated to affine Lie algebra $\widehat{sl(2)}$ at level $-1/2$. The previous theorem gives a realization of large family irreducible modules for VOA $V_{-1/2}(sl(2))$.

\begin{corollary}
Assume that $\Lambda = (\bm{\lambda}, \bm{\mu})  \ne 0$. Then $M_1(\bm{\lambda}, \bm{\mu})$ is an irreducible module for the affine Lie algebra  $\widehat{sl(2)}$  at the level $k = -\frac{1}{2}$.
\end{corollary}
\begin{proof}
For $p = 2$, $M^{\mathbb{Z}_2}$ is isomorphic to the affine VOA $V_{-\frac{1}{2}}(sl(2))$. Therefore, module $M_1(\bm{\lambda}, \bm{\mu})$ is irreducible for $V_{-\frac{1}{2}}(sl(2))$. Now Lemma \ref{ired-11} implies that $M_1(\bm{\lambda}, \bm{\mu})$ is an irreducible module for affine Lie algebra $\widehat{sl(2)}$.
\end{proof}

\begin{remark}
 The irreducible   weight modules for the Weyl vertex algebra were analysed  in \cite{AdP-2019}. One can easily show that weight modules, denoted by $\widetilde{U(\lambda)}$, have the property $\widetilde{U(\lambda)}\circ g_p \cong \widetilde{U(\lambda)}$.  Then Theorem   \ref{reducibility} implies that they  are direct sum of two irreducible relaxed weight modules for the affine vertex algebra $V_{-\frac{1}{2}}(sl(2))$ (see also \cite{A-2019}, \cite{KR-19}). 
\end{remark}



\vskip10pt {\footnotesize{}{ }\textbf{\footnotesize{}D.A.}{\footnotesize{}:
Department of Mathematics, University of Zagreb, Bijeni\v{c}ka 30,
10 000 Zagreb, Croatia; }\texttt{\footnotesize{}adamovic@math.hr}{\footnotesize \par}

\textbf{\footnotesize{}C. L.}{\footnotesize{}: Institute of Mathematics,
Academia Sinica, Taipei 10617, Taiwan; }\texttt{\footnotesize{}chlam@math.sinica.edu.tw}{\footnotesize \par}

\textbf{\footnotesize{}V.P.}{\footnotesize{}: Department of Mathematics,
University of Zagreb, Bijeni\v{c}ka 30, 10 000 Zagreb, Croatia; }\texttt{\footnotesize{}vpedic@math.hr}{\footnotesize \par}

\textbf{\footnotesize{}N.Y.}{\footnotesize{} School of Mathematical
Sciences, Xiamen University, Fujian, 361005, China;} \texttt{\footnotesize{}
ninayu@xmu.edu.cn}{\footnotesize \par}

{\footnotesize{}}}{\footnotesize \par}


\begin{thebibliography}{10}
\bibitem{A1} T. Abe, Fusion rules for the free bosonic orbifold vertex
operator algebra, J. Algebra 299 (2000), no. 1, 333\textendash 374.

\bibitem{A2}T. Abe, Fusion rules for the charge conjugation orbifold,
J. Algebra 242 (2001), no. 2, 624\textendash 655.

\bibitem{A3}T. Abe, Rationality of the vertex operator algebra $V_{L}^{+}$
for a positive definite even lattice $L$, Math. Z. 249 (2005), no. 2,
455\textendash 484.

\bibitem{AD}T. Abe, C. Dong, Classification of irreducible modules
for the vertex operator algebra $V_{L}^{+}$: general case, J. Algebra
273 (2004), no. 2, 657\textendash 685.

\bibitem{ADL} T. Abe, C. Dong, H. Li, Fusion rules for the vertex
operator $M(1)^{+}$ and $V_{L}^{+}$, Commun. Math. Phys. 253 (2005),
no. 1, 171\textendash 219.

\bibitem{A-2019}D. Adamovi\'{c}, Realizations of simple affine vertex algebras and their modules: the cases $\widehat{sl(2)}$ and $\widehat{osp(1,2)}$,  Commun. Math. Phys. 366  (2019), 1025-1067;

\bibitem{ALM} D. Adamovi\'{c}, X. Lin, A. Milas, ADE subalgebras
of the triplet vertex algebra $\mathcal{W}(p)$: D-series, Internat.
J. Math. 25 (2014), no. 1, 1450001, 34 pp.

\bibitem{ALZ} D. Adamovi\'{c},   R. Lu, K. Zhao, Whittaker modules
for the affine Lie algebra $A_{1}^{(1)}$, Adv. Math. 289 (2016),
438\textendash 479.

\bibitem{AMPP} D. Adamovi\'{c},    P. M\" oseneder Frajria, P.  Papi,    O. Per\v se, Conformal embeddings in affine vertex superalgebras, 
arXiv:1903.03794

\bibitem{AdP-2019}  D. Adamovi\' c, V. Pedi\' c, On fusion rules and intertwining operators for the Weyl vertex algebra, arXiv:1903.10248 [math.QA]

\bibitem{BM} P. Batra, V. Mazorchuk, Blocks and modules for Whittaker
pairs, \emph{ }J. Pure Appl. Algebra 215 (2011), no. 7, 1552\textendash 1568.

\bibitem{Bor} R. E. Borcherds, Vertex algebras, Kac\textendash Moody
algebras, and the Monster, Proc. Natl. Acad. Sci. USA 83 (1986), 3068\textendash 3071.

\bibitem{BO} G. Benkart, M. Ondrus, Whittaker modules for Generalized
Weyl Algebras, Represent. Theory 13 (2009), 141\textendash 164.

\bibitem{D1}C. Dong, Vertex algebras associated with even lattices,
J. Algebra 160 (1993), no.1, 245\textendash 265.

\bibitem{D2} C. Dong, Twisted modules for vertex algebras associated
with even lattices, J. Algebra 165 (1994), no. 1, 91\textendash 112.

\bibitem{DJL}C. Dong, C. Jiang, X. Lin, Rationality of vertex operator
algebra $V_{L}^{+}$ : higher rank,\emph{ }Proc. Lond. Math. Soc.
104 (2012), no. 4, 799\textendash 826.

\bibitem{DLM-1997} C. Dong, H. Li, G. Mason, Twisted representations
of vertex operator algebras, Math Ann. 310 (1998), no. 3, 571\textendash 600.

\bibitem{DM} C. Dong, G. Mason, On quantum Galois theory, Duke Math.
J. 86 (1997), no. 2, 305\textendash 321.

\bibitem{DN1} C. Dong, K. Nagatomo, Classification of irreducible
modules for vertex operator algebra $M(1)^{+}$, J. Algebra 216 (1999),
no. 1, 384\textendash 404.

\bibitem{DN2} C. Dong, K. Nagatomo, Representations of vertex operator
algebra $V_{L}^{+}$ for rank one lattice $L$,\emph{ }Comm. Math.
Phys.\emph{ }202 (1999), no. 1, 169\textendash 195.

\bibitem{DN3}C. Dong, K. Nagatomo, Classification of irreducible
modules for the vertex operator algebra $M(1)^{+}$: Higher rank,
J. Algebra 240\textbf{ }(2001), no. 1, 289\textendash 325.

\bibitem{DY} C. Dong, G. Yamskulna, 
Vertex operator algebras, generalized doubles and dual pairs, Math. Z. 241 (2002) 397\textendash423


\bibitem{F} E. Frenkel,\emph{ }Wakimoto modules, opers and the center
at the critical level,\emph{ }Adv. Math. 195 (2005), no. 2, 297\textendash 404.

\bibitem{FB} E. Frenkel, D. Ben-Zvi, Vertex algebras and algebraic
curves, Mathematical Surveys and Monographs, 88. American Mathematical
Society, Providence, RI, 2004. xiv+400 pp.

\bibitem{FF} A.J. Feingold, I.B. Frenkel, Classical affine algebras,
Adv. in Math. 56 (1985), no. 2, 117\textendash 172.

\bibitem{FGM} V. Futorny, D. Grantcharov,  V. Mazorchuk, Weight
modules over infinite dimensional Weyl algebras, Proc. Amer. Math.
Soc. 142 (2014), no. 9, 3049\textendash 3057.

 \bibitem{Wall} R. Goodman and N. R. Wallach, Symmetry, Representations, and Invariants, Graduate Texts in Mathematics, Vol. 255 (Springer, Dordrecht, 2009); doi:10.1007/978-0-387-79852-3.



\bibitem{HY} J. T. Hartwig, N. Yu, Simple Whittaker modules over
free bosonic orbifold vertex operator algebras,  Proc. Amer. Math. Soc. (to appear), arXiv:1806.06133.


\bibitem{KR} V. Kac, A. Radul, Representation theory of the vertex
algebra $W_{1+\infty}$, Transform. Groups 1 (1996), no. 1-2, 41\textendash 70.

\bibitem{KR-19} K. Kawasetsu, D. Ridout,  Relaxed highest-weight modules I: rank 1 cases. Commun. Math. Phys. (to appear), arXiv:1803.01989.

\bibitem{Li} H.S. Li, Representation theory and tensor product theory for vertex operator algebras, Ph.D. Thesis, Rutgers The State University of New Jersey - New Brunswick, 1994. 

\bibitem{L}A. Linshaw, Invariant Theory and the Heisenberg Vertex
Algebra, Int. Math. Res. Not. IMRN (2012), no. 17, 4014\textendash 4050.

\bibitem{LL}J. Lepowsky, H. Li, Introduction to vertex operator algebras
and their representations, Progress in Math. Vol. 227 Boston: Birkhauser,
2004.

\bibitem{MPS} A. Milas, M. Penn, H. Shao, Permutation Orbifolds of
the Heisenberg Vertex Algebra $\mathcal{H}(3)$, J. Math. Phys.  60 (2019), no. 2, 021703, 17 pp., arXiv:1804.01036.

\bibitem{MZ} V. Mazorchuk, K. Zhao, Simple Virasoro modules which
are locally finite over a positive part, Selecta Math. (N.S.) 20 (2014),
no. 3, 839\textendash 854.

\bibitem{T1} K. Tanabe, Simple weak modules for the fixed point subalgebra
of the Heisenberg vertex operator algebra of rank $1$ by an automorphism
of order $2$ and Whittaker vectors, Proc. Amer. Math. Soc. 145 (2017),
no. 10, 4127\textendash 4140.
\end{thebibliography}
\end{document}